\documentclass{amsart}

\usepackage[utf8]{inputenc}
\usepackage[T1]{fontenc}
\usepackage{amssymb,amsmath,amsthm,amsfonts,enumerate}
\usepackage{array}
\usepackage{tikz,tkz-tab}
\usepackage{graphicx}
\usepackage{xcolor}
\usepackage{amsrefs}
\usepackage[colorlinks, citecolor=blue, linkcolor=red, pdfstartview=FitB]{hyperref}
\usepackage{float}          % Pour figures où on veut (mettre [h] )
\usepackage{subcaption}     % Pour def environnement subfigure

\newtheorem{theorem}{Theorem}[section]
\newtheorem{lemma}[theorem]{Lemma}
\newtheorem{proposition}[theorem]{Proposition}

\newtheorem{corollary}[theorem]{Corollary}
\theoremstyle{remark}
\newtheorem{remark}[theorem]{Remark}
\theoremstyle{definition}

\newtheorem{example}[theorem]{Example}

\newcommand{\B}{\mathcal{B}}

\newcommand{\C}{\mathbb{C}}

\newcommand{\ps}[1]{\left<#1\right>}

\title{On the image of the mean transform}
\author{Fadil Chabbabi}
\address{Department of Mathematics, Faculty of Sciences, Tetouan, Abdelmalek Essaadi University, B. P. 2121 Tetouan, Morocco}
\email{f.chabbabi@uae.ac.ma}

\author{Maëva Ostermann}
\address{D\'epartement de math\'ematiques et de statistique, Universit\'e Laval,
Qu\'ebec City (Qu\'ebec),  Canada G1V 0A6.}
\email{maeva.ostermann@mat.ulaval.ca}

\subjclass[2010]{47A05, 47A10, 47B20, 47B40}

\keywords{ normal,  quasi-normal operators, polar decomposition, mean transform}
\date{2022-07-24}

\begin{document}
\maketitle
%%%%%%%%%%%%%%%%%%%%%%%%%%%%%%%%%%%%%%%%%%%%%%%%%%%%%%%%%%%%%%%%%%%%%%%%%%%%%%%%%%%%%%%%%%%%%%%%%%%%%%%%%%%%%%%%%%%%%%%%%%%%%%%%%%
\begin{abstract}
Let $\B(H)$ be the algebra of all bounded operators on a Hilbert space $H$. Let $T=V|T|$ be the polar decomposition of an operator $T\in \B(H)$. The mean transform of $T$ is defined by $M(T)=\frac{T+|T|V}{2}$. In this paper, we discuss several properties related to the spectrum, the kernel, the image, the polar decomposition of mean transform. Moreover, we investigate the image and preimage by the mean transform of some class of operators as  positive, normal, unitary, hyponormal and co-hyponormal operators.
\end{abstract}
%%%%%%%%%%%%%%%%%%%%%%%%%%%%%%%%%%%%%%%%%%%%%%%%%%%%%%%%%%%%%%%%%%%%%%%%%%%%%%%%%%%%%%%%%%%%%%%%%%%%%%%%%%%%%%%%%%%%%%%%%%%%%%%%%%
\section{Introduction}
Let $H$ be a Hilbert space with the associated inner product $\ps{\cdot,\cdot}$ and  $\B(H)$ be the algebra of all bounded and linear operators on $H$. For an operator $T\in \B(H)$, we denote by $Ker(T)$ the kernel of $T$, $Im(T)$ the range of $T$, $\sigma(T)$, $\sigma_p(T)$ and $\sigma_a(T)$ are the spectrum, the eigenvalues and the approximate spectrum  of  $T$ respectively. We say that
\begin{itemize}
    \item $T$ is \textit{binormal} if $[T^*T,TT^*]=0$ or equivalently if $T^*T$ and $TT^*$ commute.
    \item $T$ is \textit{quasi-normal} if $TT^*T=T^*TT$.
    \item $T$ is \textit{$p$-hyponormal} if $(T^*T)^p\ge(TT^*)^p$ and in particular, if $p=1$ we say that $T$ is \textit{hyponormal} and if $p=\frac12$,  $T$ is \textit{semi-hyponormal}.
    \item $T$ is \textit{$\log$-hyponormal} if $\log(T^*T)\ge\log(TT^*)$.
\end{itemize}
For  $T\in \B(H)$, we consider its polar decomposition $T=V|T|$, where $V$ is the associate partial isometry with kernel condition $Ker(V)=Ker(T)$ and $|T|=(T^*T)^{1/2}$ is the module of $T$. Then this partial isometry $V$ verifies
\[V^*V=P_{\overline{Im(|T|)}}\quad \text{and}\quad VV^*=P_{\overline{Im(T)}},\]
where $P_M$ denotes the orthogonal projection on the closed subspace $M$. It is well known that if $T=V|T|$ is the polar decomposition of $T$ then, $|T^*|=V|T|V^*$ and $T^*=V^*|T^*|$ is the polar decomposition of $T^*$. Moreover, the kernels satisfy
\[Ker(T)=Ker(|T|)=Ker(V)\quad \text{and}\quad Ker(T^*)=Ker(|T^*|)=Ker(V^*),\]
and the ranges satisfy
\[\overline{Im(T^*)}=\overline{Im(|T|)}={Im(V^*)}\quad \text{and}\quad \overline{Im(T)}=\overline{Im(|T^*|)}={Im(V)}.\]
In particular, 
$$V^*VT^*=T^*,~V^*V|T|=|T|,~V^*VV^*=V^*,~VV^*T=T~\text{and}~VV^*|T^*|=|T^*|.$$
Introduced in \cite{LeeLeeYoon2014}, the mean transform $M(T)$ of an operator $T=V|T|$ is defined by 
\[M(T)=\dfrac{T+|T|V}{2}=\dfrac{T+\tilde{T}}{2},\]
i.e. $M(T)$ is the arithmetic mean of $T$ and  the Duggal's  transformation $\tilde{T}=|T|V$ of $T$. The mean transform has been studied in detail, see
\cites{BenhidaCurtoLee2019, ChabbabiCurtoMbekhta2019, ChabbabiMbekhta2019, ChabbabiMbekhta2020,ChabbabiOstermann_preprint2022, Lee2016, LeeLeeYoon2014}.
In this paper, we study several properties related to the spectrum, the kernel, the image, the polar decomposition of mean transform. We also investigate the image and preimage of some class of operator as : positive, normal, unitary, hyponormal, semi-hyponormal and co-hyponormal operator by the mean transformation.

The paper is organized as follows. In section \ref{section:hyponormality}, we investigate the polar decomposition, the hyponormality and co-hyponormality of mean transform. In section \ref{section:spectral}, we study the relation between the joint approximate point spectrum  of an operator $T$ and $M(T)$. In section \ref{section:normal}, several results are given on the image and preimage of the mean transform. More precisely, we  give a complete characterisation of operator $T$ such that $M(T)$ is normal. We also show that $M(T)\geq 0$ if and only if $T\geq 0$. In finite dimensional case with a condition on the spectrum, we prove that $T$ is unitary operator if and only if $M(T)$ is also an unitary operator. In section \ref{section:nilpotent}, we prove that, $T^2=0$ if and only if $M(T)^2=0$. In section \ref{section:finiterank}, we give the solution of the equation $M(X)=T$, when $T$ is an operator of rank one or a normal operator  of rank two. In particular, we give an another proof of non-injectivity of  the mean transform as a mapping on $\mathcal{B}(H)$ to it self. 

%%%%%%%%%%%%%%%%%%%%%%%%%%%%%%%%%%%%%%%%%%%%%%%%%%%%%%%%%%%%%%%%%%%%%%%%%%%%%%%%%%%%%%%%%%%%%%%%%%%%%%%%%%%%%%%%%%%%%%%%%%%%%%%%%%%%%%%%%%%%%%%%%%%%%%%%%%%%%%%%%%%%%%%%%%%%%%%%%%%%%%%%%%%%%%%%%%%
\section{Polar decomposition and hyponormality of mean transform}\label{section:hyponormality}

\begin{proposition}\label{Pr:Kernels}
Let $T\in \B(H)$, and $T=V|T|$ be the polar decomposition of $T$. Then \[Ker({M}(T))=Ker(T)=Ker(V)~\text{and}~Ker(V^*)\cap Ker(V)\subseteq Ker( M(T)^*).\]
 Moreover, if $T$ is binormal then  $$Ker({M}(T)^*)=Ker(V^*).$$

\end{proposition}
 \begin{proof}
The equality  $$ Ker({M}(T))=Ker(T)=Ker(V)$$ is showed in \cite{ChabbabiCurtoMbekhta2019}. The inclusion 
$$Ker(V^*)\cap Ker(V)\subseteq Ker( M(T)^*)$$ is deduced directly from the definition of $M(T)^*$. 
 
Suppose now that $T$ is binormal. Let $x\in Ker(V^*)=Ker(|T^*|)$. Then $V^*x=|T^*|x=0$. Since $T$ is binormal, we have $$|T^*||T|x=|T||T^*|x=0.$$
We deduce that $|T|x\in Ker(V^*)$ and $2{M}(T)^*x=T^*x+V^*|T|x=0$ and then $Ker(V^*)\subset Ker(M(T)^*)$. Now let 
 $x\in Ker({M}(T)^*) $. Then 
 $$V^*(V|T|V^*x+|T|x)=|T|V^*x+V^*|T|x=0.$$
 Therefore $$|T^*|x+|T|x=V|T|V^*x+|T|x\in Ker(V^*)=Im(V)^{\perp}.$$
 Since $|T^*|x=V|T|V^*x\in Im(V)$, we have $$|T^*|x\perp |T^*|x+|T|x\in Ker(V^*),$$
and so  $$0=\ps{ |T^*|x,|T^*|x+|T|x}=\big\||T^*|x\big\|^2+\ps{ |T^*|x,|T|x}.$$
Since $T$ is binormal, i.e. $|T||T^*|=|T^*||T|$, we have $|T|^{1/2}|T^*|=|T^*||T|^{1/2}$ and thus 
\begin{align*}
\|T^*x\|^2 & = -\ps{ |T||T^*|x,x}=-\ps{ |T|^{1/2}|T^*||T|^{1/2}x,x}
\\ & =  -\ps{|T^*||T|^{1/2}x,|T|^{1/2}x } \leq 0.
\end{align*}
Then  $\|T^*x\|=0$ and $T^*x=0$. Finally, we deduce that $Ker(M(T)^*)\subset Ker(T^*)$.
\end{proof}
\begin{proposition}\label{Prop:PolarDecompM(T)}
Let $T\in B(H)$ and $T=V|T|$ be the polar decomposition of $T$. If $Ker(T^*)\subset Ker(T)$ then
\[M(T)=V\left(\frac{|T|+V^*|T|V}2\right)\quad \text{and}\quad M(T)^*=V^*\left(\frac{|T|+V|T|V^*}2\right),\]
are the polar decomposition of $M(T)$ and $M(T)^*$ respectively.
\end{proposition}
\begin{proof} Since $Ker(T^*)\subseteq Ker(T)$, we deduce that 
\[\overline{Im(|T|)}=\overline{Im(T^*)}=Ker(T)^\perp\subseteq Ker(T^*)^\perp =\overline{Im(T)}.\]
Then
\[|T|=P_{\overline{Im(T)}}|T|=VV^*|T|.\]
It implies that
\begin{align*}
4M(T)^*M(T)
&=(T^*+V^*|T|)(T+|T|V)\\
&=|T|^2+V^*|T|^2V+|T|V^*|T|V+V^*|T|V|T|\\
&=|T|^2+(V^*|T|V)^2+|T|V^*|T|V+V^*|T|V|T|\\
&=(|T|+V^*|T|V)^2.
\end{align*}
Hence
\[|M(T)|=\frac12(|T|+V^*|T|V).\]
Moreover
\[M(T)=\frac12(T+|T|V)=\frac12(V|T|+VV^*|T|V)=V\left(\frac{|T|+V^*|T|V}2\right).\]
By Proposition, we have \ref{Pr:Kernels} $Ker(M(T))=Ker(V)$, and then the preceding equality is the polar decomposition of $M(T)$.

On the other hand, we know that if $M(T)=V|M(T)|$ is the polar decomposition of $M(T)$ then $M(T)^*=V^*|M(T)^*|$ is the polar decomposition of $M(T)^*$, with 
\[|M(T)^*|=V|M(T)|V^* =\frac12(V|T|V^*+VV^*|T|VV^*)=\frac12(V|T|V^*+|T|).\]
Then, we deduce directly  \[M(T)^*=V^*\left(\dfrac{|T|+V|T|V^*}2\right)\]
is the polar decomposition of $M(T)^*$.
\end{proof}
\begin{remark}
If $T$ is $p$-co-hyponormal, then the condition $Ker(T^*)\subset Ker(T)$ is satisfied and the results of preceding proposition hold. 
If $T$ is $p$-hyponormal, we have only the inclusion $Ker(T)\subset Ker(T^*)$. 
\end{remark}

\begin{theorem}\label{T:semi-hyponormal}
Let $T\in\B(H)$ and $T=V|T|$ the polar decomposition  of $T$, and suppose that $Ker(T)=Ker(T^*)$. Consider the following statements:
\begin{enumerate}
\item $T$ is a semi-hyponormal,
\item $V|T|V^*\leq |T|\leq V^*|T|V$,
\item $\tilde{T}$ is semi-hyponormal,
\item $M(T)$ is semi-hyponormal,
\item $\Delta(T)$ is hyponormal,
\item $V|T|V^*\leq V^*|T|V$.
\end{enumerate}
Then we have 
\[(1)\iff(2)\iff(3)\implies(4)\iff(5)\iff(6).\]
\end{theorem}
\begin{lemma}
Let $T\in\B(H)$ and $T=V|T|$ be its polar decomposition. Suppose that  $Ker(T^*)\subseteq Ker(T)$ Then $|\tilde{T}|=V^*|T|V$ and $|\tilde{T}^*|=|T|$. 
\end{lemma}
\begin{remark}
If $Ker(T)= Ker(T^*)$, we even have that $\tilde T^*=V^*|T|$ is the polar decomposition of $\tilde T^*$.
\end{remark}
\begin{proof}
From the assumption  $Ker(T^*)\subseteq Ker(T)$, we deduce  $\overline{Im(|T|)}=\overline{Im(T^*)}\subset \overline{Im(T)}$. Then $VV^*|T|=P_{\overline{Im(T)}}|T|=|T|$. We deduce that 
\[\tilde T^*\tilde T=V^*|T|^2V=V^*|T|VV^*|T|V=(V^*|T|V)^2.\]
Then $|\tilde T|=V^*|T|V$.

In the same way, we have 
$|\tilde{T}^*|^2=\tilde{T}\tilde{T}^*=|T|VV^*|T|=|T|^2,$
and then $|\tilde{T}^*|=|T|.$
\end{proof}
\begin{proof}[Proof of Theorem \ref{T:semi-hyponormal}]
The equivalence $(1)\iff(2)$ is deduced directly from the definition and the fact that $VV^*|T|=|T|$. 

$(2)\iff (3)$ is the fact that  $|\tilde{T}|=V^*|T|V$ and $|\tilde{T}^*|=|T|$. 

$(2)\;\implies (6)$ is clear.

$(4)\iff(6)$ is deduced directly from Proposition. \ref{Prop:PolarDecompM(T)}.

$(6)\;\implies(5)$ Suppose that $V^*|T|V\ge|T^*|=V|T|V^*$. Then 
$$|T|^{\frac12}V^*|T|V|T|^{\frac12}\ge |T|^{\frac12}V|T|V^*|T|^{\frac12}.$$
Then $\Delta(T)$ is hyponormal.

$(5)\implies(6)$ Suppose that $\Delta(T)$ is hyponormal. Then $\Delta(T)\Delta(T)^*\ge\Delta(T)^*\Delta(T)$. Since $|T|^{1/2}$ is self-djoint, we have
\begin{multline*}
|T||T^*||T|=|T|V|T|V^*|T|=|T|^{1/2}\Delta(T)\Delta(T)^*|T|^{1/2}
\\\le|T|^{1/2}\Delta(T)^*\Delta(T)|T|^{1/2}=|T|V^*|T|V|T|=|T||\tilde{T}||T|.
\end{multline*}
Then, we deduce that \[|T^*|\le|\tilde{T}|~\text{on}~Im(|T|).\]
Thus 
\[|{T}^*|=V^*V|T^*|V^*V\le V^*V|\tilde{T}|V^*V=|\tilde{T}|.\]
\end{proof}

With a similar proof, we have the following result.
\begin{theorem}\label{T:semicohyponormal}
Let $T\in\B(H)$ and $T=V|T|$ the polar decomposition  of $T$, and suppose that $Ker(T^*)\subset  Ker(T)$. Consider the following statements:
\begin{enumerate}
\item $T$ is a semi co-hyponormal;
\item $V|T|V^*\geq |T|\geq V^*|T|V$;
\item $\tilde{T}$ is semi co-hyponormal;
\item $M(T)$ is semi co-hyponormal;
\item $V|T|V^*\geq V^*|T|V$;
\item $\Delta(T)$ is co-hyponormal.
\end{enumerate}
Then we have 
\[(1)\iff(2)\iff(3)\implies(4)\iff(5)\implies(6).\]
If moreover $Ker(T^*)=Ker(T)$ then $(6)\implies(5)$.
\end{theorem}
\begin{remark}
The condition  $Ker(T^*)\subseteq Ker(T)$ is a direct consequence of $(1)$ and, in particular, it is not necessary to add this hypothesis to prove that $(1)\implies(5)$.
\end{remark}
 \begin{proof} $(1)\iff(2)\iff(3)\iff(4)$ are direct consequences of the definition.
 
 $(1)\implies(5)$ Suppose that $T$ is semi-co-hyponormal. Then $|T|\leq |T^*|=V|T|V^*$ and  $Ker(T^*)\subseteq Ker(T)$. Hence  $$2M(T)=V|T|+|T|V=V(|T|+V^*|T|V)$$
    and 
    $$2M(T)^*=|T|V^*+V^*|T|=V^*(V|T|V^*+|T|)$$
   are  the polar decomposition of $M(T)$ and $M(T)^*$ respectively. On the other hand, since  $V^*|T|V\leq |T|\leq V|T|V^*$, then
 \begin{eqnarray*}
2|M(T)| & = & |T|+V^*|T|V 
 \leq  2|T|\\
&\leq &  |T|+V|T|V^*=2|M(T)^*|.
 \end{eqnarray*}
 Therefore 
 $$|M(T)|\leq |T|\leq |M(T)^*|,$$
 and thus $M(T)$ is semi-co-hyponormal. 
 
$(5)\implies(6)$ Suppose that $M(T)$ is semi-co-hyponormal, i.e. 
 $|M(T)|\leq |M(T)^*|$. From the condition $Ker(T^*)\subseteq Ker(T)$, we deduce that
$$2|M(T)|=|T|+V^*|T|V\;\; \text{ and }\;\; 2|M(T)^*|=V|T|V^*+|T|.$$
And then, the assertion (6) holds. 
 \end{proof}
 \begin{theorem}\label{T:cohyponormal}
 If $T$ is co-hyponormal and binormal, then $M(T)$ is co-hyponormal.
 \end{theorem}

 \begin{proof}
Since $T$ is co-hyponormal, we have that $$|T|^2=T^*T\leq TT^*=|T^*|^2=(V|T|V^*)^2$$ and, in particular, 
$Ker(T^*)\subseteq Ker(T)$. Moreover, by taking the square root, we deduce that $|T|\le V|T|V^*$, and by multiplying by $V^*$ on the left and by $V$ on the right, we have also $V^*|T|V\le V^*V|T|V^*V=|T|$.

Moreover, by Proposition \ref{Prop:PolarDecompM(T)}, the polar decompositions of $M(T)$ and $M(T)^*$ are given by
$$M(T)=V\left(\dfrac{|T|+V^*|T|V}{2}\right)\;\; \text{ and } \;\; M(T)^*=V^*\left(\dfrac{V|T|V^*+|T|}{2}\right).$$
Therefore
\begin{eqnarray*}
4M(T)^*M(T)
&=&(|T|+V^*|T|V)^2\\
&=&|T|^2+|T|V^*|T|V+V^*|T|V|T|+(V^*|T|V)^2\\
&=&|T|^2+|T|V^*|T|V+V^*|T|V|T|+V^*|T|VV^*|T|V\\
&=&|T|^2+|T|V^*|T|V+V^*|T|V|T|+V^*|T|^2V.
\end{eqnarray*}
With the same calculus, we have,
\begin{eqnarray*}
4M(T)M(T)^*&=&|T|^2+|T|V|T|V^*+V|T|V^*|T|+V|T|^2V^*.
\end{eqnarray*}
Moreover, since  $T$ is binormal, then 
$$|T|V|T|V^*=|T||T^*|=|T^*||T|=V|T|V^*|T|.$$
By multiplying this equation on the left by $V^*$ and on the right by $V$, we have  
$$V^*|T|V|T|=|T|V^*|T|V.$$
Thus, from the preceding equations, we conclude that 
\begin{eqnarray*}
4M(T)M(T)^*-4M(T)^*M(T)  & \geq & 
|T||T^*|+|T^*||T|-|T|V^*|T|V-V^*|T|V|T|\\
& = &  |T|\left(|T^*|-V^*|T|V\right)+(|T^*|-V^*|T|V)|T|\\
& = &  2|T|^{\frac{1}{2}}(V|T|V^*-V^*|T|V)|T|^{\frac{1}{2}}\geq 0,
\end{eqnarray*}
which implies that $M(T)$ is co-hyponormal. 
\end{proof}
%%%%%%%%%%%%%%%%%%%%%%%%%%%%%%%%%%%%%%%%%%%%%%%%%%%%%%%%%%%%%%%%%%%%%%%%%%%%%%%%%%%%%%%%%%%%%%%%%%%%%%%%%%%%%%%%%%%%%%%%%%%%%%%%%%%%%%%%%%%%%%%%%%%%%%%%%%%%%%%%%%%%%%%%%%%%%%%%%%%%%%%%%%%%%%%%%%%
%%%%%%%%%%%%%%%%%%%%%%%%%%%%%%%%%%%%%%%%%%%%%%%%%%%%%%%%%%%%%%%%%
\section{Spectral properties of the mean transform}\label{section:spectral}
For a operator $T\in\B(H)$, the \textit{approximate point spectrum $\sigma_a(T)$ of $T$} is the set of the complex numbers $\lambda\in\C$ such that there exists a sequence of unit vectors $(x_n)$ in $H$ satisfying \[Tx_n-\lambda x_n\to0.\] The \textit{joint approximate point spectrum $\sigma_{aj}(T)$ of $T$} is the set of the complex numbers $\lambda\in\C$ such that there exists a sequence of unit vectors $(x_n)$ in $H$ verifying 
\[Tx_n-\lambda x_n\to0\quad \text{and}\quad T^*x_n-\bar\lambda x_n\to0.\]
It is well known that  \[\partial \sigma(T)\subseteq \sigma_a(T)\subseteq \sigma(T).\]  Moreover, we have  $\sigma_{aj}(T)\subseteq \sigma_a(T)$, but in general the inclusion is strict. In finite dimensional case  $\sigma_a(T)$ and $\sigma(T)$ coincide with the set of eigenvalues of $T$. If $T$ is either normal, hyponormal, $p$-hyponormal or $\log$-hyponormal then 
$\sigma_{aj}(T)=\sigma_a(T)$, see \cite{Aluthge2002} for details. The following result characterizes the joint approximate point spectrum.   
\begin{proposition}[Lemma 2.4 \cite{Xia1983}]\label{T:Xia}

Let $T=V|T|$ be the polar decomposition of an operator $T\in \B(H)$ and  $\lambda=|\lambda|e^{i\theta} \in \C^*$. Then 
 $\lambda \in \sigma_{aj}(T)$ if and only if there exists a sequence of unit vectors $(x_n)$, such that 
 $$|T|x_n-|\lambda|x_n\longrightarrow 0, \;\;\; |T^*|x_n-|\lambda|x_n\longrightarrow 0 , $$ 
 and 
 $$Vx_n-e^{i\theta}x_n\longrightarrow 0, \;\;\; V^*x_n-e^{-i\theta} x_n\longrightarrow 0.$$ 
\end{proposition}
As a consequence of this result, we deduce the following proposition.
\begin{proposition}\label{Pr:ajSpectrum}
Let $T\in \B(H)$. Then $\sigma_{aj}(T)\subseteq \sigma_{aj}(M(T)).$

 Moreover if $T$ is semi-hyponormal and $Ker(T)=Ker(T^*)$, then we have 
 $$\sigma_{aj}(T)\backslash \{0\}=\sigma_{aj}(M(T))\backslash \{0\}.$$ 
\end{proposition}

\begin{proof}
Let $\lambda\in \C^*$ and  $\lambda \in \sigma_{aj}(T)$. Let $(x_n)$ be a sequence of unit vectors such that 
 $$|T|x_n-|\lambda|x_n\longrightarrow 0, \;\;\; |T^*|x_n-|\lambda|x_n\longrightarrow 0 , $$ 
 and 
 $$Vx_n-e^{i\theta}x_n\longrightarrow 0, \;\;\; V^*x_n-e^{-i\theta} x_n\longrightarrow 0.$$ 
Then we have
\begin{eqnarray*}
 M(T)x_n-\lambda x_n & = &\dfrac{Tx_n-\lambda x_n+|T|Vx_n-\lambda x_n}{2} \\
 & = & \dfrac{Tx_n-\lambda x_n+|T|Vx_n-e^{i\theta }|T|x_n+e^{i\theta }|T|x_n-e^{i\theta }|\lambda| x_n}{2} \\
  & = & \dfrac{Tx_n-\lambda x_n+|T|(Vx_n-e^{i\theta }x_n)+e^{i\theta }(|T|x_n-|\lambda| x_n)}{2}.
\end{eqnarray*}
Since $$ Tx_n-\lambda x_n\longrightarrow 0,\quad |T|(Vx_n-e^{i\theta }x_n)\longrightarrow 0\quad\text{and}\quad 
|T|x_n-|\lambda| x_n\longrightarrow 0,$$
 we deduce that
\[
 M(T)x_n-\lambda x_n \longrightarrow 0. 
\]
With a similar way, we get also 
\[
 M(T)^*x_n-\overline{\lambda} x_n \longrightarrow 0. 
\]
This implies that $\lambda \in \sigma_{aj}(T)$. 

Now, suppose that $T$ is semi-hyponormal and $Ker(T^*)=Ker(T)$. By Proposition \ref{Prop:PolarDecompM(T)},
\[M(T)=V\left(\frac{|T|+V^*|T|V}{2}\right)~\text{and}~M(T)^*=V^*\left(\frac{|T|+V|T|V^*}{2}\right)\] are the polar decomposition of $M(T)$ and $M(T)^*$. Let $\lambda=|\lambda|e^{i\theta}\ne 0$ belong in the joint-spectrum of $M(T)$. Then there exists a sequence $(x_n)$ of unit vector in $H$, such that 
$$2|M(T)|x_n-2|\lambda|x_n\underset{n\to +\infty}{\longrightarrow } 0, ~~ Vx_n-e^{i\theta}x_n\underset{n\to +\infty}{\longrightarrow }  0,$$
$$2|M(T)^*|x_n-2|\lambda|x_n\underset{n\to +\infty}{\longrightarrow } 0 ~ \text{and}~ V^*x_n-e^{-i\theta}x_n\underset{n\to +\infty}{\longrightarrow }  0.$$
Then  
$$|T|x_n+V^*|T|Vx_n-2|\lambda|x_n\underset{n\to +\infty}{\longrightarrow } 0 ~ \text{and}~|T|x_n+V|T|V^*x_n-2|\lambda|x_n\underset{n\to +\infty}{\longrightarrow }0.$$
Now, since $T$ is semi-hyponormal, the inequalities  
$V|T|V^*\leq |T|\leq V^*|T|V$ hold, which implies that 
$$ |T|+V|T|V^*-2|\lambda|\leq 2|T|-2|\lambda|\leq |T|+ V^*|T|V-2|\lambda|.$$
Then 
$$|T|x_n-|\lambda|x_n\underset{n\to +\infty}{\longrightarrow } 0.$$
And finally we deduce that $\lambda\in \sigma_{aj}(T)$.
\end{proof}
%%%%%%%%%%%%%%%%%%%%%%%%%%%%%%%%%%%%%%%%%%%%%%%%%%%%%%%%%%%%%%%%%%%%%%%%%%%%%%%%%%%%%%%%%%%%%%%%%%%%%%%%%%%%%%%%%%%%%%%%%%%%%%%%%%%%%%%%%%%%%%%%%%%%%%%%%%%%%%%%%%%%%%%%%%%%%%%%%%%%%%%%%%%%%%%%%%%%%%%%%%%%%%%%%%%%%%%%%%%%%%%%%%%%%%%%%%%%%%%%%%%%%%%%%%%%%%%%%%%%

\section{The mean transform and normal operators} \label{section:normal}
If $T$ is quasi-normal (or normal), then ${M}(T)=T$ and, in particular, $M(T)$ is also quasi-normal (or normal). But the converse is not true in general. For example, let $T=\begin{pmatrix}
1 & 1\\
-1 & -2
       \end{pmatrix}$, then $|T|=\begin{pmatrix}
1 & 1\\
1 & 2
       \end{pmatrix}$, and $V=\begin{pmatrix}
1 & 0\\
0 & -1
       \end{pmatrix}$. Thus, we have
$${M}(T)=\dfrac{T+|T|V}{2}=\begin{pmatrix}
1 & 0\\
0 & -2
\end{pmatrix}\;\; \text{is self-adjoint}.$$
In \cite{ChabbabiCurtoMbekhta2019}, the authors characterized the operators with a self-adjoint mean transform by obtaining that  $M(T)$ is self-adjoint if and only if $V^*=V$.
The following theorem gives an another equivalent condition. 
%%%%%%%%%%%%%%%%%%%%%%%%%%%%%%%%%%%%%%%%%%%%%%%%%%%%%%%%%%%%%%%%%
\begin{theorem}\label{T:M(T)auto-adj}
Let $T\in B(H)$ and $T=V|T|$ be its polar decomposition. Then the following assertions are equivalent.
\begin{enumerate}
    \item $V=V^*$;
    \item $M(T)=M(T^*)$;
    \item $M(T)=M(T)^*$.
\end{enumerate}
\end{theorem}
\begin{proof}(1)$\iff$(3) It is proved in \cite{ChabbabiMbekhta2020}.

(1)$\implies$(2) If $V=V^*$ then 
$$\overline{Im(T)}=\overline{Im(T^*)}=\overline{Im(|T|)}\;\;\text{ and }\;\; V^*V=VV^*=V^2=(V^*)^2.$$ 
Thus $V^2$ is the orthogonal projection on $\overline{Im(|T|)}$. Hence 
$$M(T^*)=\dfrac{T^*+|T^*|V^*}{2}=\dfrac{|T|V^*+V|T|(V^*)^2}{2}=\dfrac{|T|V+V|T|}{2}=M(T).$$

(2)$\implies$(1) Suppose that $M(T)=M(T^*)$ then we have 
\[
T+|T|V=T^*+|T^*|V^*=|T|V^*+V|T|(V^*)^2
\]
By taking the adjoint,  we get 
\[
V^*(V|T|V^*+|T|)=|T|V^*+V^*|T|=V|T|+V^2|T|V^*=V(|T|+V|T|V^*).
\]
Then the operators $V^*=V$ on $Im(|T|+V|T|V^*)$. Since 
$$Ker(|T|+V|T|V^*)=Ker(|T|)\cap Ker(V|T|V^*)=Ker(V)\cap Ker(V^*),$$
we have $V^*=V=0$ on $Ker(|T|+V|T|V^*)$. Therefore the equality  $V^*=V$ holds  on the Hilbert space $H$.  
\end{proof} 
Note that Theorem \ref{T:M(T)auto-adj} implies that if one of the equivalent assertions is satisfied for $T$ then we have $M(T^*)=M(T)^*$. But this condition is not equivalent to any  assertion in this theorem thanks to the result of the following  theorem.

\begin{theorem}\label{T:M(T)normal}
Let  $T=V|T|$ be the polar decomposition of the operator $T\in B(H)$. Suppose that $Ker(T)=Ker(T^*)$. Then the following assertions are equivalent.
\begin{enumerate}
    \item[(1)] $M(T)$ is normal;
    \item[(2)] $V^*|T|V=V|T|V^*$;
    \item[(3)] $V^2|T|=|T|V^2$;
    \item[(4)] $M(T^*)=M(T)^*$.
\end{enumerate}
\end{theorem}
Note that the implication $(4)\implies (3)$ is true without the kernel condition and the implication $(3)\implies (4)$ holds also if $ Im(T)$ is dense.
\begin{proof}
Since $Ker(T)=  Ker(T^*)$ then $V$ is normal, so we have $VV^*=V^*V.$\\
%%%%%%%%%%%%%%%%%%%%%%%%%%%%%%%%%%%%%%%%%%%%%%%%%%%%%%%%%%%%%%%

$(1)\iff(2)$ 
By Proposition \ref{Prop:PolarDecompM(T)}, we have $$\big| M(T)\big|=\dfrac{|T|+V^*|T|V}{2} \;\; \text{and }\;\; \big| M(T)^*\big|=\dfrac{V|T|V^*+|T|}{2}.$$ 
Now, we have $ M(T)$ is a normal operator if and only if 
 $\big| M(T)\big|=\big| M(T)^*\big|$ and this is equivalent to $V|T|V^*=V^*|T|V$. 
 
%%%%%%%%%%%%%%%%%%%%%%%%%%%%%%%%%%%%%%%%%%%%%%%%%%%%%%%%%%%%%%%
$(2)\implies(3)$ If  $V|T|V^*=V^*|T|V$ then, by multiplying on the left and the right this equation by $V$ and since $V$ is normal, we have
 \[V^2|T|=V^2|T|V^*V=VV^*|T|V^2=|T|V^2.\]
 
 $(3)\implies (2)$ If $V^2|T|=|T|V^2$ then \[(V^*)^2|T|VV^*=(V^*)^2|T|=|T|(V^*)^2\] and thus $(V^*)^2|T|V=|T|V^*$ on $Im(V^*)$. Since $Ker(V)=Ker(V^*)$, it is clear that this equality is also true on $Im(V^*)^\perp=Ker(V)$. Then \[(V^*)^2|T|V=|T|V^*~\text{on}~H.\]
By taking the adjoint, we deduce that \[V^*|T|V^2=V^2|T|=V|T|V^*V.\] Then $V^*|T|V=V|T|V^*$ on $Im(V)$ but it is true on $Im(V)^\perp=Ker(V^*)=Ker(V)$ and thus 
\[V^*|T|V=V|T|V^*~\text{on}~H.\]
%%%%%%%%%%%%%%%%%%%%%%%%%%%%%%%%%%%%%%%%%%%%%%%%%%%%%%%%%%%%%%%%
$(3)\implies (4)$ If $|T|V^2=V^2|T|=V^2|T|V^*V$ then $|T|V=V^2|T|V^*$ on $Im(V)$. So if $Ker(T)=Ker(T^*)$, then it is clear that this equality is true on $Ker(T^*)=Im(V)^\perp$ and thus $|T|V=V^2|T|V^*$ on $H$.

%%%%%%%%%%%%%%%%%%%%%%%%%%%%%%%%%%%%%%%%%%%%%%%%%%%%%%%%%%%%%%%%
(4)$\implies(3)$ Recall that if $T=V|T|$, then the polar decomposition of $T^*$ is given by $T^*=V^*|T^*|$ with $|T^*|=V|T|V^*$. Therefore
\begin{align*}
M(T^*)=M(T)^*
\iff & T^*+|T^*|V^*=T^*+V^*|T|\\
\iff &V^*|T|=|T^*|V^*\\
\iff &V^*|T|=V|T|(V^*)^2\\
\iff &|T|V=V^2|T|V^*\\
\implies&|T|V^2=V^2|T|V^*V=V^2|T|.
\end{align*}
%%%%%%%%%%%%%%%%%%%%%%%%%%%%%%%%%%%%%%%%%%%%%%%%%%%%%%%%%%%%%%%%

\end{proof}
\begin{corollary}\label{C:normal}
Let $T\in B(H)$ and $T=V|T|$ be its polar decomposition. Suppose that $T$ is a finite rank operator or with a dense range. If $M(T)$ is normal then $V$ is a normal partial isometry and  $V|T|V^*=V^*|T|V$. In particular, $V^2$ and $|T|$ commute. 
\end{corollary}
\begin{proof}
By Proposition \ref{Pr:Kernels}, 
 we have 
 $$ Ker(T)= Ker( M(T))= Ker( M(T)^*).$$
Then,
\begin{align*}
Ker(T) = Ker(T)\cap  Ker( M(T)^*) = & Ker(|T|)\cap Ker( M(T)^*)\\
 \subseteq & Ker(2M(T)^*-V^*|T|)= Ker(T^*).
\end{align*}
If $T$ is a finite rank operator (and then $T^*$ is also) or if $T$ has a dense range, the inclusion $Ker(T)\subset Ker(T^*)$ gives finally that 
  $$ Ker(T)=  Ker(T^*)~\text{and}~ Im(T)=  Im(T^*).$$
  In particular, these equalities imply that $VV^*=V^*V$, according to the projections $V^*V$ and $V^*V$ have the same kernel and range. Finally, by Theorem \ref{T:M(T)normal}, we deduce that $V|T|V^*=V^*|T|V$.
\end{proof}

%%%%%%%%%%%%%%%%%%%%%%%%%%%%%%%%%%%%%%%%%%%%%%%%%%%%%%%%%%%%%%%%%%%%%%%%%%%%%%%%%%%%%%%%%%%%%%%%%%%%%%%%%%%%%%%%%%%%%%%%%%%%%%%%%%%%%%%%%%%%%%%%%%%%%%%%%%%%%%%%%%%%%%%%%%%%%%%%%%%%%%%%%%%%%%%%%%%%%%%%%%%%%%%%%%%%%%%%%%%%%%%%%%%%%%%%%%%%%%%%%%%%%%%%%%%%%%%%%%%%%%%%%%%%%%%%%%%%%%%%%%%%%%%%%%%%%%%%%%%%%%%%%%%%%%%%%%%%%%%%%%%%

 In the rest of this section, we consider particular cases of normal operators: the positive operators and the unitary matrices.
\begin{theorem}
Let $T\in \B(H)$ and $T=V|T|$ be the polar decomposition of $T$,  then the following equivalence holds :
$${M}(T) \geq 0 \iff \; T \geq 0.$$
\end{theorem}
\begin{proof}
If $T\ge0$ then $M(T)=T$ and thus it is clear that $M(T)$ is also a positive operator. Now, suppose that $M(T)\ge0$.  Then in particular ${M}(T)$ is self-adjoint and then by Theorem \ref{T:M(T)auto-adj}, the partial isometry $V=V^*$ is self-adjoint. Therefore 
$${M}(T)=V\left (\dfrac{|T|+V|T|V}{2}\right)=\left (\dfrac{|T|+V|T|V}{2}\right)V\geq 0.$$  
Let $A=\dfrac{|T|+V|T|V}{2}$. Then we have  $VA=AV\geq 0$ and 
$$Ker(A)=Ker(|T|)=Ker(V),\;\;\overline{Im(A)}=\overline{Im(|T|)}=\overline{Im(V)}.$$
Moreover, $AVA\geq 0$  as  the product of  two commuting and positive operators. Therefore the restriction $V :{\overline{Im(A)}}\to \overline{Im(A)} $ is positive. Since $V|_{\overline{Im(A)}^{\perp}}=0$, this implies that $V$ is positive on $H$. Since the positive operators have a unique positive square root, we deduce that $V$ is the orthogonal projection on $\overline{Im(|T|)}$ and then $T=V|T|=|T|\geq 0$.
\end{proof} 
% We finish this section by considering matrices having a unitary mean transform.
  \begin{theorem}\label{T:M(T)unitary}
 Let $T\in M_n(\mathbb C)$ and  $T=V|T|$ be its polar decomposition. Suppose that the spectrum of $ V $ does not contain two opposite values i.e. if $\lambda \ne \mu$ two eigenvalues  of  $V$ then $\lambda+\mu \ne 0$. Then 
 $$ M(T) \;\text{ is an unitary matrix}\;\; \iff \;\; T=V \;\text{ is also}. $$
 \end{theorem}
Note that the hypothesis on the spectrum of $V$ is important. In general, an operator can have unitary mean transform without being a unitary operator itself. Indeed, let 
  $$T=\begin{pmatrix}
0 & \frac{3}{2}\\
\frac{1}{2} & 0
       \end{pmatrix},\;\;  \text{then} \;\; |T|=\begin{pmatrix}
\frac{1}{2} & 0\\
0 & \frac{3}{2}
       \end{pmatrix}, \;\text{ and } \;\; V=\begin{pmatrix}
0 & 1\\
1 & 0
       \end{pmatrix}.$$
Thus
       $${M}(T)=\dfrac{T+|T|V}{2}=\begin{pmatrix}
0 & 1\\
1 & 0
       \end{pmatrix}=V\;\; \text{is unitary}.$$
An another consequence of this example is that the mean transform  ${M}:\B(H)\to \B(H)$ is not one-to-one.
 \begin{proof}[Proof of Theorem \ref{T:M(T)unitary}] 
 If $T$ is a unitary matrix then it is clear that $M(T)=T$ and, in particular, $M(T)$ is also unitary. Now, suppose that $ M(T)$ is unitary operator. Then, in particular, it is a normal operator. By Corollary \ref{C:normal}, the polar part $V$ is normal. Since $ Ker( M(T))= Ker(V)=\{0\}$, we deduce that $V$ is an unitary operator, i.e. $V^*V=VV^*=I$. Moreover, since $M(T)$ is unitary, we have
$$ I= M(T)^* M(T)=\left(\frac{|T|+V^*|T|V}{2}\right )V^*V\left(\frac{|T|+V^*|T|V}{2}\right )=\left(\frac{|T|+V^*|T|V}{2}\right )^2.$$
Therefore, by taking the square root, we deduce that $ |T|+V^*|T|V=2I.$

 Now since $V\in M_n(\mathbb C)$ is a unitary matrix, then  $V$ is diagonalizable in orthonormal basis $\mathcal{\varepsilon}=(e_i)_{i=1\dots n}$ and its spectrum satisfies $\sigma(V)=\{\mu_1\dots \mu_n\}\subset\mathbb{T}$, with $Ve_i=\mu_ie_i$ for $i=1\dots n$.
 Since $ |T|+V^*|T|V=2I$, we have, for all $i=1\dots n$,
$$|T|e_i+\mu_i V^*|T|e_i=2e_i$$
and then
\begin{eqnarray*} 
2&=& \ps{ 2 e_i, e_i} \\
&= & \ps{ |T|e_i+V^*|T|Ve_i,e_i}\\
&= & \ps{ |T|e_i+\mu_iV^*|T|e_i,e_i}\\
&= & \ps{ |T|e_i,e_i} +\mu_i\ps{ V^*|T|e_i,e_i}\\
%&= & \ps{ |T|e_i,e_i}+\mu_i \ps{ |T|e_i,Ve_i}\\
&= &\ps{ |T|e_i,e_i}+\mu_i \ps{ |T|e_i,\mu_ie_i}\\.
%&=&\ps{ |T|e_i,e_i}+\mu_i\bar{\mu}_i \ps{ |T|e_i,e_i}\\
&=&\ps{ |T|e_i,e_i}+|\mu_i|^2 \ps{ |T|e_i,e_i}\\
&=&2\ps{ |T|e_i,e_i}.
\end{eqnarray*}
Then, for any $i=1\dots n$, we have  $\ps{ |T|e_i,e_i}=1$.
 
On the other hand, for $i\ne j$, 
\begin{eqnarray*} 
0&=& \ps{ 2 e_i, e_j} \\
&= & \ps{ |T|e_i+V^*|T|Ve_i,e_j}\\
&= & \ps{ |T|e_i+\mu_iV^*|T|e_i,e_j}\\
&= & \ps{ |T|e_i,e_j} +\mu_i\ps{ V^*|T|e_i,e_j}\\
&= & \ps{ |T|e_i,e_j}+\mu_i \ps{ |T|e_i,Ve_j}\\
&= &\ps{ |T|e_i,e_j}+\mu_i \ps{ |T|e_i,\mu_je_j}\\.
&=&\ps{ |T|e_i,e_j}+\mu_i\bar{\mu_j} \ps{ |T|e_i,e_j}\\
%&=&\ps{ |T|e_i,e_j}+|\mu_i|^2 \ps{ |T|e_i,e_i}\\
&=&\ps{ |T|e_i,e_j}\big(1+\mu_i\bar{\mu_j}\big).
\end{eqnarray*}
Then $$\ps{ |T|e_i,e_i}\big(1+\mu_i\bar{\mu_j}\big)=0.$$
Since the spectrum of $ V $ does not contain two opposite values, we deduce that 
$1+\mu_i\bar{\mu_j}\ne 0$, and thus $\ps{ |T|e_i,e_j}=0$.

 We conclude that $$\ps{ |T|e_i,e_j}=\delta_{i,j}\;\; \text{ for}\;\; i,j=1\dots n.$$
Finally, we deduce that $|T|=I$ and then $T=V$ is unitary. 
\end{proof}
%%%%%%%%%%%%%%%%%%%%%%%%%%%%%%%%%%%%%%%%%%%%%%%%%%%%%%%%%%%%%%%%%%%%%%%%%%%%%%%%%%%%%%%%%%%%%%%%%%%%%%%%%%%%%%%%%%%%%%%%%%%%%%%%%%%%%%%%%%%%%%%%%%%%%%%%%%%%%%%%%%%%%%%%%%%%%%%%%%%%%%%%%%%%%%%%%%%
\section{The mean transform and 2-nilpotent operator}\label{section:nilpotent}
The mean transform gives a good description of $2$-nilpotent operator. In \cite{ChabbabiCurtoMbekhta2019}, the authors proved that an operator $T$ is $2$-nilpotent (i.e. $T^2=0$) if and only if $M(T)=T/2$. This result allows us to characterize when $M(T)$ is $2$-nilpotent. Precisely, we have the following result.
\begin{theorem}Let $T\in \B(H)$ and $T=V|T|$ be the polar decomposition of $T$. Then
$${M}(T)^2=0\;\;\iff \;\; T^2=0.$$
\end{theorem} 
 \begin{proof}
If $T^2=0$ then ${M}(T)=T/2$ is also of square zero. Now, suppose that $M(T)^2=0$. Then we have 
 \begin{equation}\label{eq2}
 4M(T)^2=(V|T|+|T|V)(V|T|+|T|V)=0
 \end{equation}
 Multiply the equation on left and the right by $V^*$, it follows that
 $$V^*(V|T|+|T|V)(V|T|+|T|V)V^*=(|T|+V^*|T|V)(V|T|V^*+|T|VV^*)=0.$$
 Take the adjoint, then we have 
 $$(V|T|V^*+VV^*|T|)(|T|+V^*|T|V)=0.$$
Then the operator $V|T|V^*+VV^*|T|=0$ on the image of $|T|+V^*|T|V$. Moreover, since $$Ker(|T|+V^*|T|V)=Ker(|T|)=Ker(V),$$
then $$\overline{Im(|T|+V^*|T|V)}=\overline{Im(|T|)}=\overline{Im(V^*)}.$$
Therefore the operator $V|T|V^*+VV^*|T|=0$ on $Im(V^*)$, it means that 
\begin{align*}
V^*(V|T|V^*+VV^*|T|)V^*
&=V^*(V|T|(V^*)^2+VV^*|T|V^*)   \\
&=V^*V|T|(V^*)^2+V^*VV^*|T|V^*  \\
&=P_{\overline{Im(|T|)}}|T|(V^*)^2+V^*P_{\overline{Im(T)}}|T|V^*\\
&=|T|(V^*)^2+V^*|T|V^*=0.
\end{align*}
Then $|T|(V^*)^2+V^*|T|V^*=0$, and if we multiply on the right by $V^2$ this equation, we have 
\begin{multline*}
0=(|T|(V^*)^2+V^*|T|V^*)V^2=|T|(V^*)^2V^2+V^*|T|V^*V^2\\
=|T|(V^*)^2V^2+V^*|T|P_{\overline{Im(|T|)}}V=|T|(V^*)^2V^2+V^*|T|V.
\end{multline*}
This equality implies that $|T|(V^*)^2V^2=-V^*|T|V$ which is a self-adjoint operator and, in particular, $|T|(V^*)^2V^2=-V^*|T|V\le0$. But, $|T|(V^*)^2V^2=(V^*)^2V^2|T|$ then it is the product of two commuting and positive operators so $|T|(V^*)^2V^2=(V^*)^2V^2|T|\ge0$. Then we deduce that $V^*|T|V=-|T|(V^*)^2V^2=0$. Now let $S=|T|^{1/2}V$. Then 
\[
S^*S=V^*|T|^{1/2}|T|^{1/2}V=V^*|T|V=0,\] 
and thus $S=0$. Finally, we deduce that
$$T^2=V|T|V|T|=V|T|^{1/2}S|T|=0.$$
\end{proof}
%%%%%%%%%%%%%%%%%%%%%%%%%%%%%%%%%%%%%%%%%%%%%%%%%%%%%%%%%%%%%%%%%%%%%%%%%%%%%%%%%%%%%%%%%%%%%%%%%%%%%%%%%%%%%%%%%%%%%%%%%%%%%%%%%%%%%%%%%%%%%%%%%%%%%%%%%%%%%%%%%%%%%%%%%%%%%%%%%%%%%%%%%%%%%%%%%%%%%%%%%%%%%%%%%%%%%%%%%%%%%%%%%%%%%%%%%%%%%%%%%%%%%%%%%%%%%%%%%%%
\section{Some other equations with the mean transform}\label{section:finiterank}

In this section, we will discuss the following problem: for $T\in \B(H)$, find a operator $X$ such that 
\begin{equation}\label{eq3} 
M(X)=T
\end{equation} 
We have already consider the cases of positive operators, self-adjoint operators or 2-nilpotent operators. In the following examples, we consider some rank one or two operators.

\begin{example}
If $T=x \otimes y$ is a rank one operator,  if $X$ is a solution of $M(X)=T$, then $$Ker(X)=Ker(M(X))=Ker(T)=span\{ y \}^{\perp},$$
and thus $Im(X^*)=Im(T^*)=span\{ y \}$, hence $X$ and $X^*$ are also of rank one, write $X=z\otimes y$ and thus, 
$$M(X)=\frac{1}{2}\Big(z+\frac{\ps{ z, y}}{\|y\|^2}y\Big)\otimes y=x \otimes y.$$
Then 
$z+\frac{\ps{ z, y}}{\|y\|^2}y=2x$ and 
$\ps{ z, y}=\ps{ x, y}$, it follows that 
$z=2x-\frac{\ps{ x, y}}{\|y\|^2}y$.
\end{example}
\begin{example}
If $T=\delta P + \nu Q$ with $P=x\otimes x$,   $Q=y\otimes y$ are two rank one and orthogonal projections, such that $PQ=0$ and $\delta, \nu\in \C^*$. It is clear that  $T$  is  normal and 
$V=\frac{\delta}{|\delta|}P+\frac{\nu}{|\nu|}Q=e^{i\theta}P+e^{i\varphi}Q$ is the polar part of $T$. Now let $X=V_X|X|$ Such that $M(X)=T$, then in particular $M(X)$ is normal and so by Proposition \ref{Prop:PolarDecompM(T)}, 
$V_X=V=e^{i\theta}P+e^{i\varphi}Q$. By the assumption, we have
\[
V|X|+|X|V=2(\delta P + \nu Q).
\]
Then
\[
|X|+V^*|X|V=2V^*(\delta P + \nu Q)=2(|\delta|P+|\nu|Q).
\]
If we apply this equation to $x$, it follows that 
$$|X|x+e^{i\theta}V^*|X|x=2|\delta|x.$$
We take the scalar product with $x$ and $y$ respectively to obtain that
\begin{equation}\label{eq6}
2\ps{ |X|x,x}=2|\delta|\;\;\text{ and } \;\; \ps{ |X|x,y}\big(1+e^{i\theta}e^{-i\varphi}\big)=0.
\end{equation}
Let us consider two cases that ar $1+e^{-\theta}e^{-i\varphi}$ is null or not.\\
{\bf{Case 1:}} Suppose that $ e^{i\theta}\ne-e^{i\varphi}$ Then, the equation \ref{eq6} becomes 
$$\ps{ |X|x,x}=|\delta|\;\;\text{ and } \;\; \ps{ |X|x,y}=0.$$
With the same steps, we also have that
$$\ps{ |X|y,y}=|\nu|.$$
This implies that $$|X|x=|\delta|x~\text{and}~|X|y=|\nu|y,$$
and then 
$$|X|=|\delta|P+|\nu|Q~\text{and}~ X=V|X|=\delta P+ \nu Q =T$$
is the unique solution. 

{\bf{Case 2:}} Suppose that $ e^{i\theta}=-e^{i\varphi}$. In this case, we have 
$|X|x=|\delta|x+\beta y$ and $|X|y= \overline{\beta}x+|\nu| y$ with $|\delta\nu|-|\beta|^2\geq 0$, i.e. we can write
$$|X|=\begin{pmatrix}
 |\delta| & \overline{\beta} \\ 
\beta  &  |\nu|
       \end{pmatrix}, ~ V=e^{i\theta}\begin{pmatrix}
 1 & 0 \\ 
0  &  -1
       \end{pmatrix}~\text{and}~ T=\begin{pmatrix}
 \delta & 0 \\ 
0  &  \mu
       \end{pmatrix}$$
We deduce that
$$X=e^{i\theta}\begin{pmatrix}
  |\delta| & \overline{\beta} \\ 
-{\beta}  &  -|\mu|
       \end{pmatrix}.$$ %%%%%%Il y aurait pas des confusions entre les \nu et les \mu dans la preuve de cet exemple?
and for this $X$, we actually have $M(X)=T$ since
\begin{align*}
    2M(X) &=  e^{i\theta}\begin{pmatrix}
 1 & 0 \\ 
0  &  -1
       \end{pmatrix} \begin{pmatrix}
 |\delta| & \overline{\beta} \\ 
\beta  &  |\nu|
       \end{pmatrix}+ \begin{pmatrix}
 |\delta| & \overline{\beta} \\ 
\beta  &  |\nu|
       \end{pmatrix}e^{i\theta}\begin{pmatrix}
 1 & 0 \\ 
0  &  -1
       \end{pmatrix}\\
&=  e^{i\theta} \begin{pmatrix}
 2 |\delta|  & 0 \\ 
0  &  -2 |\nu|
       \end{pmatrix} =2T.
\end{align*}
\end{example}
\bibliographystyle{plain}
\bibliography{biblio}
\end{document}